\documentclass{amsart}

\usepackage{amssymb,latexsym,amsmath,extarrows}
\usepackage{graphicx}

\newtheorem*{acknowledgement}{Acknowledgement}
\newtheorem{theorem}{Theorem}[section]
\newtheorem{lemma}[theorem]{Lemma}
\newtheorem{proposition}[theorem]{Proposition}
\newtheorem{remark}[theorem]{Remark}

\newtheorem{definition}[theorem]{Definition}

\newtheorem{conjecture}[theorem]{Conjecture}

\newcommand{\wh}{\widehat}

\newcommand{\ZR}{\mathbb{R}}

\newcommand{\bT}{{\bf T}}

\begin{document}

\title{ $L^p$-estimates of maximal function related to  Schr\"{o}dinger Equation in $\mathbb{R}^2$}
\author{Xiumin Du}
\address{Mathematics Department\\
University of Illinois at Urbana-Champaign\\
Urbana, IL 61801}
\email{xdu7@illinois.edu}
\author{Xiaochun Li}
\address{Mathematics Department\\
University of Illinois at Urbana-Champaign\\
Urbana, IL 61801}
\email{xcli@math.uiuc.edu}
\date{\today}

\begin{abstract}
Using Guth's polynomial partitioning method, we obtain $L^p$ estimates
for the maximal function associated to the solution of Schr\"odinger equation in $\mathbb R^2$.
The $L^p$ estimates can be used to
recover the previous best known result that $\lim_{t \to 0} e^{it\Delta}f(x)=f(x)$ almost
everywhere for all $f \in H^s (\mathbb{R}^2)$ provided that $s>3/8$.
\end{abstract}

\maketitle

\section{Introduction} \label{intro}
\setcounter{equation}0

The solution to the free Schr\"{o}dinger equation
\begin{equation}
  \begin{cases}
    iu_t - \Delta u = 0, &(x,t)\in \mathbb{R}^n \times \mathbb{R} \\
    u(x,0)=f(x), & x \in \mathbb{R}^n
  \end{cases}
\end{equation}
is given by
$$
  e^{it\Delta}f(x)=(2\pi)^{-n}\int e^{i\left(x\cdot\xi+t|\xi|^2\right)}\widehat{f}(\xi) \, d\xi.
$$
We use $B(c, r)$ to represent a ball centered at $c$ with radius $r$ in $\mathbb R^2$.
The main theorem in this article is the following:
\begin{theorem}\label{thm-3}
For $2\leq p \leq 3.2$, for any $\epsilon >0$, there exists a constant $C_\epsilon$ such that
\begin{equation}\label{max4}
 \big\|\underset{0<t\leq R}{\text{sup}}|e^{it\Delta}f|\big\|_{L^p(B(0,R))} \leq
C_\epsilon R^{\frac{2}{p}-\frac{5}{8}+\epsilon} \|f\|_2,
\end{equation}
 holds for all $R\geq 1$ and all $f$ with ${\rm supp}\widehat{f}\subset A(1)=\{\xi: |\xi|\sim 1\}$.
\end{theorem}

\begin{remark}
The local bound (\ref{max4}) can be used to derive immediately global estimates in $L^{p}(\mathbb R^2)$
 for $\sup_{0<t\leq 1}|e^{it\Delta }f|$, following from Theorem 10 in \cite{Rogers}.
We are indebted to K. Rogers for pointing this out to us.
\end{remark}

An interesting and important problem in PDE
is to determine the optimal $s$, for which $\lim_{t \to 0}e^{it\Delta}f(x)=f(x)$ almost everywhere whenever
$f\in H^s(\mathbb{R}^n).$ This problem originates from  Carleson \cite{lC}, who proved convergence
for $s \geq 1/4$ when $n=1$. Dahlberg and Kenig \cite{DK} showed that the convergence does not hold for $s<1/4$
in any dimension.  Sj\"{o}lin \cite{pS} and Vega \cite{lV} proved independently the convergence for $s>1/2$ in all dimensions.
However, the pointwise convergence also holds when $s>s_0$ for some $s_0<1/2$. For instance, some
positive partial results were obtained by Bourgain \cite{jB}, Moyua-Vargas-Vega \cite{MVV}, and Tao-Vargas \cite{TV}.
Lee \cite{sL} used Tao-Wolff's bilinear restriction method to get $s>3/8$ for $n=2$.
Recently Bourgain \cite{jB12}, via Bourgain-Guth's multilinear restriction method,  proved that $s>1/2-1/(4n)$ is a sufficient condition
for the pointwise convergence when $n\geq 2$, and surprisely he also had shown that $s\geq1/2-1/n$ is a necessary condition for $n\geq4$.
In the two dimensinal case, Bourgain's result coincides with Lee's. An improved necessary condition for the pointwise
convergence in $\mathbb R^n$ with $n\geq 3$ is $s\geq 1/2-1/(n+2)$ due to Luc\'a and Rogers \cite{L-R}. Most recently Bourgain \cite{jB16} proved a new necessary condition, that is, $s\geq \frac {1}{2} -\frac{1}{2(n+1)}$ is required for pointwise convergence in $\ZR^n$ with $n\geq 2$. \\

Theroem \ref{thm-3} can be used to recover the following Bourgain-Lee's pointwise convergence result in two
dimensional case.

\begin{theorem}\label{Thm:PointConv}
 For every  $f\in H^s(\mathbb R^2)$ with $s>3/8$, $\lim_{t \to 0}e^{it\Delta}f(x)=f(x)$ almost everywhere.
\end{theorem}

To see why this is true. First,
it is routine and standard that Theorem \ref{Thm:PointConv} is a consequence of the boundedness of the associated maximal
function, i.e. for some $p >0$,
\begin{equation}\label{max0}
  \big\|\underset{0<t\leq 1}{\text{sup}}|e^{it\Delta}f|\big\|_{L^p(B(0,1))} \leq C\|f\|_{H^s},
\end{equation}
holds for all $f\in H^s(\mathbb{R}^2)$. From the definition of $H^s$ space, it is
clear that (\ref{max0}) can be reduced to show that there exists some $p >0$ such that for any $R\geq 1$ and any $\epsilon >0$,
\begin{equation}\label{max1}
  \big\|\underset{0<t\leq 1}{\text{sup}}|e^{it\Delta}f|\big\|_{L^p(B(0,1))} \leq C_\epsilon R^{s+\epsilon} \|f\|_2,
 \end{equation}
holds for every $L^2$ function $f$ whose Fourier transform is supported in
$A(R)=\{\xi:|\xi|\sim R\}$. Here the constant $C_\epsilon$ is independent of $R$ and $f$.
For $p\geq 2$, it was observed by S. Lee in \cite{sL}, via a use of wave packets decomposition,
that in order to prove (\ref{max1}), it suffices to show that for any $R\geq 1$ and any $\epsilon >0$,
\begin{equation}\label{max2}
  \big\|\underset{0<t\leq 1/R}{\text{sup}}|e^{it\Delta}f|\big\|_{L^p(B(0,1))} \leq C_\epsilon R^{s+\epsilon}
\|f\|_2, \quad \forall f \,\, {\rm with}\,\, \text{supp}\,\widehat{f}\subseteq A(R)\,.
\end{equation}
 By a parabolic rescaling,  (\ref{max2}) is equivalent to
\begin{equation}\label{eq:goal}
  \big\|\underset{0<t\leq R}{\text{sup}}|e^{it\Delta}f|\big\|_{L^p(B(0,R))} \leq
C_\epsilon R^{s-1+\frac{2}{p}+\epsilon} \|f\|_2,
\end{equation}
for any $f$ with  $\text{supp}\,\widehat{f}\subseteq A(1)=\{\xi:|\xi|\sim 1\}$.\\
Because of the equivalence of (\ref{max2}) and (\ref{eq:goal}), it is clear that Theorem \ref{thm-3} implies
Theorem  \ref{Thm:PointConv}.\\

It is natural to expect the following conjecture would be true.
\begin{conjecture}
 \begin{equation}\label{max}
 \big\|\underset{0<t\leq R}{\text{sup}}|e^{it\Delta}f|\big\|_{L^p(B(0,R))} \leq
C_\epsilon R^{\epsilon} \|f\|_2,
\end{equation}
 holds for any $p\geq 3$, all $R\geq 1$ and all $f$ with ${\rm supp}\widehat{f}\subset A(1)$
\end{conjecture}

 In \cite{lG} and \cite{lG16}, Guth applied the idea of polynomial partitioning from incidence geometry to
 restriction estimates. The proof of Theorem \ref{thm-3} is based on Guth's polynomial partitioning method developed in \cite{lG} and \cite{lG16}.

For estimates of $\|e^{it\Delta}f\|_{L^p_{x,t}(B(0,R)\times [0,R])} $, we can use parabolic rescaling to reduce a linear estimate to a bilinear one. But for the $L^pL^\infty$-norm, or more generally mixed $L^pL^q$-norm, the parabolic rescaling does not work well when the Fourier support is not centered at $0$. The rescaling would change the mixed norm if the location of the Fourier support was kept in the unit ball, otherwise it would change the location of the Fourier support if the mixed norm was kept. To deal with this issue, we take the size of the Fourier support into consideration and do induction on different scales. But the smaller size of the Fourier support would cause poor separability between the two terms in bilinear estimate. So we use the following $k$-broadness concept -- $BL_{k,A}^pL^\infty$ , which is motivated by Guth \cite{lG16}. Here is the setup.

Consider functions $f$ with Fourier support $B(\xi_0,M^{-1})$, where $\xi_0 \in B(0,1)$ and $M\geq 1$, we decompose $B(\xi_0, M^{-1})$ into balls $\tau$ of radius $(KM)^{-1}$, where $K$ is a large constant. We have that $f=\sum_\tau f_\tau$, where $\widehat{f_\tau}=\widehat{f}|_\tau$. Denote $G(\tau):=\{G(\xi)\, | \,\xi\in\tau\}$, where
\[
  G(\xi):=\frac{(-2\xi,1)}{|(-2\xi,1)|}\in S^2.
\]
The set of directions $G(\tau)\subset S^2$ is a spherical cap with radius $\sim (KM)^{-1}$. If $V\subset \ZR^2 \times \ZR$ is a subspace, then we write $\text{Angle}(G(\tau),V)$ for the smallest angle between any non-zero vectors $v\in V$ and $v'\in G(\tau)$.

Next we decompose $B(0,R)$ into balls $B_K$ of radius $K$, and decompose $[0,R]$ into intervals $I_K$ of length K. For a parameter $A$, we define
\begin{equation}\label{mu}
  \mu_{e^{it\Delta}f}(B_K\times I_K):= \underset{V_1, \cdots, V_A}{\text{min}}\left( \underset{ \tau \notin V_a\,\text{for all} \, a }{\text{max}} \int_{B_K\times I_K} |e^{it\Delta}f_\tau(x)|^p\right)
\end{equation}
where $V_1,\cdots, V_A$ are $(k-1)$-subspaces of $\ZR^3$, and $\tau \notin V_a$ means that $\text{Angle}(G(\tau),V_a)>(KM)^{-1}$. Next for any subset $U\subset B^*_R := B(0,R)\times [0,R]$, we define
\begin{equation}\label{infty}
  \|e^{it\Delta}f\|_{BL_{k,A}^p L^\infty (U)}^p :=\sum_{B_k\subset B(0,R)}\underset{I_K\subset[0,R]}{\text{max}} \frac{|U\cap (B_K\times I_K)|}{|B_K\times I_K|}\mu_{e^{it\Delta}f}(B_K\times I_K).
\end{equation}

 In Section \ref{ktoreg}, we'll show that Theorem \ref{thm-3} follows  from the following theorem:
\begin{theorem}\label{klinfty}
Fix $k=2$. For $1\leq p \leq 3.2$, for any $\epsilon >0$, there exists a large constant $A$ and a constant $C(\epsilon,K)$ such that
\begin{equation}\label{eqklinfty}
 \|e^{it\Delta}f\|_{BL_{k,A}^p L^\infty (B_R^*)} \leq
C(\epsilon,K) R^{\frac{2}{p}-\frac{5}{8}+\epsilon} \|f\|_2,
\end{equation}
 holds for $\forall R\geq 1$, $\forall \xi_0 \in B(0,1), \forall M\geq 1$ and  $\forall f$ with ${\rm supp}\widehat{f}\subset B(\xi_0, M^{-1})$.
\end{theorem}

 In order to apply polynomial partitioning method, we approximate the maximum with respect to $t$ by $l^q$-norm. We define
\begin{equation}\label{q}
  \|e^{it\Delta}f\|_{BL_{k,A}^p L^q (U)}^p :=\sum_{B_k\subset B(0,R)}\left[\sum_{I_K\subset[0,R]} \left(\frac{|U\cap (B_K\times I_K)|}{|B_K\times I_K|}\mu_{e^{it\Delta}f}(B_K\times I_K)\right)^q\right]^{1/q}.
\end{equation}
Note that we have
\[
   \|e^{it\Delta}f\|_{BL_{k,A}^p L^\infty (B^*_R)}=
  \lim_{q\to \infty}\|e^{it\Delta}f\|_{BL_{k,A}^p L^q (B^*_R)} .
\]
 For later reference, for each $\epsilon>0$, we choose small parameters $0<\delta \ll \delta_2 \ll \delta_1 \ll \epsilon.$
 To prove Theorem \ref{klinfty}, it is enough to prove the following theorem:
 \begin{theorem} \label{klq}
  Fix $k=2$. For $1\leq p \leq 3.2$, for any $\epsilon >0$, there exists a large constant $A$ and a constant $C(\epsilon,K)$ such that for any $ q>\delta^{-1}$,
\begin{equation}\label{}
 \|e^{it\Delta}f\|_{BL_{k,A}^p L^q (B_R^*)} \leq
C(\epsilon,K) R^{\frac{2}{p}-\frac{5}{8}+\epsilon} \|f\|_2,
\end{equation}
 holds for $\forall R\geq 1$, $\forall \xi_0 \in B(0,1), \forall M\geq 1$ and  $\forall f$ with ${\rm supp}\widehat{f}\subset B(\xi_0, M^{-1})$.
 \end{theorem}

 \section{$2$-Broad Maximal Estimate Implies Regular Maximal Estimate}\label{ktoreg}
 \setcounter{equation}0

 In this section, we assume that Theorem \ref{klinfty} holds and we prove Theorem \ref{thm-3}. Fix $k=2$. We consider functions $f$ with ${\rm supp}\widehat{f}\subset B(\xi_0,M^{-1})$ with arbitrary $\xi_0\in B(0,1)$ and $M\geq 1$. Because ${\rm supp}\widehat{f}\subset B(\xi_0,M^{-1})\subset \ZR^2$, we have that
 \[
   \|e^{it\Delta}f_\tau\|_\infty \lesssim M^{-1}\|f\|_2,
 \]
 by interpolating this $L^\infty$ extimate with \eqref{eqklinfty}, we get
 \begin{equation}\label{eqklinfty2}
 \|e^{it\Delta}f\|_{BL_{k,A}^p L^\infty (B_R^*)} \lesssim_{K,\epsilon}
M^{-\epsilon^2}R^{\frac{2}{p}-\frac{5}{8}+\epsilon} \|f\|_2,
\end{equation}
for $1\leq p\leq 3.2$. For the rest of the argument, we fix $2\leq p\leq 3.2$.

Let $\beta$ be the best constant satisfying
\begin{equation} \label{beta}
  \|e^{it\Delta}f\|_{L^p L^\infty (B_R^*)} \lesssim_\epsilon
M^{-\epsilon^2}R^{\beta+\epsilon} \|f\|_2,
\end{equation}
for all functions $f$ with ${\rm supp}\widehat{f}\subset B(\xi_0,M^{-1})$ with arbitrary $\xi_0\in B(0,1)$ and $M\geq 1$.

We write $ \|e^{it\Delta}f\|_{L^p L^\infty (B_R^*)}^p$ as
\[
  \sum_{B_K\subset B(0,R)}\int_{B_K} \underset{I_K\subset [0,R]}{{\rm max}}\, \underset{t\in I_K}{{\rm sup}} |e^{it\Delta}f(x)|^p \,dx
\]
For each $B_K\times I_K$, we fix a choice of $(k-1)$-subspaces $V_1,\cdots,V_A$ achieving the minimum in the definition of $\mu_{e^{it\Delta}f}(B_K\times I_K)$. On $B_K\times I_K$, we bound $|e^{it\Delta}f|^p$ by
\[
  K^{O(1)}\underset{\tau\notin V_a \text{for  all} \, a}{{\rm max}}|e^{it\Delta}f_\tau|^p
  + C \sum_{a=1}^A \left|\sum_{\tau\in V_a} e^{it\Delta}f_\tau\right|^p,
\]
so we can break $ \|e^{it\Delta}f\|_{L^p L^\infty (B_R^*)}^p$ into two parts correspondingly. The first part is
\begin{equation}\label{I}
K^{O(1)} \sum_{B_K\subset B(0,R)} \underset{I_K\subset [0,R]}{{\rm max}}\,\underset{\tau\notin V_a \text{for  all} \, a}{{\rm max}}\int_{B_K\times I_K}|e^{it\Delta}f_\tau|^p\,dxdt
\end{equation}
where we use the fact that each $|e^{it\Delta}f_\tau|$ is approximately constant on $B_K\times I_K$. Now by the choice of $V_1,\cdots,V_A$ for each $B_K\times I_K$ and Theorem \ref{klinfty}, the first part \eqref{I} is bounded by
\[
  \left[C(K,\epsilon)M^{-\epsilon^2}R^{\frac{2}{p}-\frac{5}{8}+\epsilon}\|f\|_2\right]^p.
\]
Next the second part is
\begin{equation}\label{II}
  C\sum_{a=1}^{A} \int_{B(0,R)} \underset{t\in [0,R]}{\rm sup} \left|\sum_{\tau\in V_a} e^{it\Delta}f_\tau\right|^p \,dx.
\end{equation}
Note that the balls $\tau$ are disjoint with radius $(KM)^{-1}$, and each $V_a$ is a $1$-dimensional subspace ($k=2$), so the number of $\tau\in V_a$ is $O(1)$. Hence we bound \eqref{II} by
\[
CA \sum_{\tau} \int_{B(0,R)} \underset{t\in [0,R]}{\rm sup} |e^{it\Delta}f_\tau|^p \,dx,
\]
and by the definition of $\beta$\,-- \eqref{beta}, this is further bounded by
\[
  \left[ CA^{\frac 1 p}K^{-\epsilon^2}C_\epsilon M^{-\epsilon^2}R^{\beta+\epsilon}\|f\|_2\right]^p.
\]
We can choose large constants $A=A(\epsilon)$ and $K=K(\epsilon)$ with the relation $A\ll K^{\epsilon^2}$. Then the second part is done by induction. This completes the proof of Theorem \ref{thm-3}, under the assumption that Theorem \ref{klinfty} holds.

\section{Properties of $BL^p_{k,A}L^q$} \label{property}
 \setcounter{equation}0

Now our goal is to prove Theorem \ref{klq}. First let us see some properties of $BL^p_{k,A}L^q$. Recall the setup in Section \ref{intro} and the definition of $\|e^{it\Delta}f\|_{BL_{k,A}^p L^q (U)}$ in \eqref{mu} and \eqref{q}, where $U\subset B^*_R$ is a subset.

\begin{lemma}
(a) Given subsets $U_1$ and $U_2$ in $B^*_R$, we have that
\[
  \|e^{it\Delta}f\|_{BL_{k,A}^p L^q (U_1 \cup U_2)}^p \leq \|e^{it\Delta}f\|_{BL_{k,A}^p L^q (U_1)}^p +\|e^{it\Delta}f\|_{BL_{k,A}^p L^q ( U_2)}^p.
\]

(b) Suppose that $A=A_1+A_2$, where $A,A_1, A_2$ are non-negative integers, then
\[
 \|e^{it\Delta}(f+g)\|_{BL_{k,A}^p L^q (U)}^p \lesssim_p \|e^{it\Delta}f\|_{BL_{k,A_1}^p L^q (U)}^p+\|e^{it\Delta}g\|_{BL_{k,A_2}^p L^q (U)}^p.
\]

(c) Suppose that $1\leq p < r$, then
\begin{equation} \label{holder}
\|e^{it\Delta}f\|_{BL_{k,A}^p L^q (B^*_R)} \lesssim_K R^{(2+\frac 1 q)(\frac 1 p-\frac 1 r)}\|e^{it\Delta}f\|_{BL_{k,A}^r L^q (B^*_R)}
\end{equation}
\end{lemma}

\begin{proof}
Part (a) follows from the definition of $BL_{k,A}^p L^q$, the triangle inequality for $l^q$, and the simple fact that
\[
|(U_1\cup U_2)\cap(B_K\times I_K)| \leq |U_1\cap(B_K\times I_K)|+| U_2\cap(B_K\times I_K)| .
\]

Part (b) follows from the definition of $BL_{k,A}^p L^q$, the triangle inequality for $l^q$ and the following inequalities
\begin{align}
 &\underset{V_1,\cdots,V_A}{{\rm min}} \left(\underset{\tau \notin V_a ,1\leq a\leq A}{{\rm max}} \int_{B_K\times I_K} |e^{it\Delta}(f+g)_\tau|^p \right) \notag\\
 \lesssim_p & \underset{V_1,\cdots,V_A}{{\rm min}} \left(\underset{\tau \notin V_a , 1\leq a\leq A}{{\rm max}} \int_{B_K\times I_K} |e^{it\Delta}f_\tau|^p  +  |e^{it\Delta}g_\tau|^p \right) \notag\\
 \leq & \underset{V_1,\cdots,V_{A_1}}{{\rm min}}
 \left(\underset{\tau \notin V_a , 1\leq a\leq A_1}{{\rm max}}
 \int_{B_K\times I_K} |e^{it\Delta}f_\tau|^p \right) \notag \\
 +&\underset{V_{A_1+1},\cdots,V_A}{{\rm min}} \left(\underset{\tau \notin V_a , A_{1}+1\leq a\leq A}{{\rm max}} \int_{B_K\times I_K} |e^{it\Delta}g_\tau|^p \right). \notag
\end{align}

For part (c), the left-hand side is
\[
  \left\{ \sum_{B_K\subset B_R} \left[\sum_{I_K\subset [0,R]}\left(\underset{V_1,\cdots,V_A}{{\rm min}} \underset{\tau \notin V_a }{{\rm max}} \int_{B_K\times I_K} |e^{it\Delta}f_\tau|^p\right)^q\right]^{\frac 1 q}\right\}^{\frac 1 p}.
\]
First, apply the H\"older's inequality to the inner integral
\[
   \int_{B_K\times I_K} |e^{it\Delta}f_\tau|^p \leq  \left(\int_{B_K\times I_K} |e^{it\Delta}f_\tau|^r \right)^{\frac p r} K^{3(1-\frac p r)}.
\]
Next by applying H\"older's inequality to the sum over $I_K\subset [0,R]$ and then to the sum over $B_K \subset B_R$, we bound $\|e^{it\Delta}f\|_{BL_{k,A}^p L^q (B^*_R)}$ by
\[
 \left[\sum_{B_K\subset B_R}\left(\sum_{I_K\subset [0,R]} K^{3q}\right)^{\frac 1 q}\right]^{\frac 1 p - \frac  1 r} \cdot  \|e^{it\Delta}f\|_{BL_{k,A}^r L^q (B^*_R)},
\]
which is $C(K)R^{(2+\frac 1 q)(\frac 1 p -\frac 1 r)} \|e^{it\Delta}f\|_{BL_{k,A}^r L^q (B^*_R)}.$
\end{proof}

\section{Polynomial Partitioning}
\setcounter{equation}0

In this section, we work in general dimension--$\ZR^n\times \ZR$. We aim to introduce a polynomial $P$ in the polynomial ring $\mathbb {R}[x,t]$ such  that
the variety $Z(P)=\{(x,t)\in \mathbb{R}^n \times \mathbb{R}: P(x,t)=0\}$
 bisects  every member in a collection of some quantities. It relies on the famous Borsuk-Ulam Theorem,  asserting that
 {\it if $F:\mathbb{S}^N \xrightarrow{} \mathbb{R}^N$ is a continuous function,  where $\mathbb{S}^N$ is the $N$-sphere, then there exists a point $v\in \mathbb{S}^N$ with $F(v)=F(-v)$.}

First we state a sandwich theorem, which is a consequence of Borsuk-Ulam Theorem.

\begin{lemma} \label{Thm:Sandwich}
Suppose that $f_1,f_2, \cdots, f_N$ are functions with ${\rm supp}\widehat{f_j}\subset B(0,1)\subset \ZR^n$, $U_1,U_2,\cdots,U_N$ are subsets of $B^n(0,R)\times [0,R]$, and $1\leq p, r<\infty$, then there exists a non-zero polynomial $P$ on $\mathbb{R}^n \times \mathbb{R}$ of degree $\leq c_nN^{1/(n+1)}$ such that for each $1\leq j\leq N$ we have
  \[
    \|e^{it\Delta}f_j\|_{BL_{k,A}^p L^r (U_j\cap \{P>0\})}^p=\|e^{it\Delta}f_j\|_{BL_{k,A}^p L^r (U_j\cap \{P<0\})}^p.
  \]
\end{lemma}

\begin{proof}
  Let $V$ be the vector space of polynomials on $\mathbb{R}^n \times \mathbb{R}$ of degree at most $D$, then
\[
  \text{Dim} V = \binom{D+n+1}{n+1} \sim_n D^{n+1}.
\]
So we can choose $D \sim N^{1/(n+1)}$ such that $\text{Dim} V\geq N+1$, and without loss of generality we can assume $\text{Dim} V = N+1$ and identify $V$ with $\mathbb{R}^{N+1}$. We define a function $G$ as follows:
\begin{align}
  \mathbb{S}^N \subseteq V \backslash \{0\} & \xlongrightarrow{G} \mathbb{R}^N \notag\\
  P   & \mapsto \{G_j(P)\}_{j=1}^{N} \notag
\end{align}
where
\[
  G_j(P):= \|e^{it\Delta}f_j\|_{BL_{k,A}^p L^r (U_j\cap \{P>0\})}^p-\|e^{it\Delta}f_j\|_{BL_{k,A}^p L^r (U_j\cap \{P<0\})}^p ,
\]
it is obvious that $G(-P)=-G(P).$
Assume that the function $G$ is continuous, then Borsuk-Ulam Theorem implies that there exists $P\in \mathbb{S}^N \subseteq V
\backslash \{0\}$ with $G(P)=G(-P)$, hence $G(P)=0$, and $P$ obeys the conclusion of Lemma \ref{Thm:Sandwich}.
It remains to check the continuity of the functions $G_j$ on $V \backslash \{0\}$.

 Suppose that $P_l \to P$ in $V \backslash \{0\}$. Note that
\[
  |G_j(P_l)-G_j(P)| \leq 2 \|e^{it\Delta}f_j\|_{BL_{k,A}^p L^r (U_j\cap \{P_lP\leq 0\})}^p
\]
while $P_l \to P$ implies that
\[
  \bigcap_{l_0}\bigcup_{l\geq l_0}\{(x,t): P_l(x,t)\cdot P(x,t)\leq0\} \subseteq P^{-1}(0),
\]
so we have that
\[
  \lim_{l_0\to\infty} \|e^{it\Delta}f_j\|_{BL_{k,A}^p L^r (U_j\cap (\cup_{l\geq l_0}\{P_lP\leq 0\}))}^p
  = \|e^{it\Delta}f_j\|_{BL_{k,A}^p L^r (U_j\cap P^{-1}(0))}^p=0,
\]
the last equality follows from $|P^{-1}(0)|=0$. This proves that $\lim_{l\to\infty}|G_j(P_l)-G_j(P)| =0$, showing that $G_j$ is continuous on $V\backslash\{0\}.$
\end{proof}

Next we use Lemma \ref{Thm:Sandwich} repeatedly to prove the following partitioning result.
\begin{theorem} \label{Thm:PolynPart}
Suppose that $f$ is a function with ${\rm supp}\widehat{f}\subset B(0,1)\subset \ZR^n$, $U$ is a subset of $B^n(0,R)\times [0,R]$, and $1\leq p, r<\infty$, then for each $D$ there exists a non-zero polynomial $P$ of degree at most D such that $(\mathbb{R}^n\times\mathbb{R})\backslash Z(P)$ is a union of $\sim_n D^{n+1}$ disjoint open sets $O_i$ and for each $i$ we have
  \[
  \|e^{it\Delta}f\|_{BL_{k,A}^p L^r (U)}^p \leq c_n D^{n+1}\|e^{it\Delta}f\|_{BL_{k,A}^p L^r (U\cap O_i)}^p.
  \]
\end{theorem}

\begin{proof}
  By Lemma \ref{Thm:Sandwich}, we obtain a polynomial $P_1$ of degree $\lesssim 1$ such that
  \[
    \|e^{it\Delta}f\|_{BL_{k,A}^p L^r (U\cap \{P_1>0\})}^p=\|e^{it\Delta}f\|_{BL_{k,A}^p L^r (U\cap \{P_1<0\})}^p.
  \]
Next by Lemma \ref{Thm:Sandwich} again we obtain a polynomial $P_2$ of degree $\lesssim 2^{1/(n+1)}$ such that
\begin{align}
  \|e^{it\Delta}f\|_{BL_{k,A}^p L^r (U\cap \{P_1>0\}\cap \{P_2>0\})}^p &=\|e^{it\Delta}f\|_{BL_{k,A}^p L^r (U\cap \{P_1>0\}\cap \{P_2<0\})}^p\,, \notag\\
  \|e^{it\Delta}f\|_{BL_{k,A}^p L^r (U\cap \{P_1<0\}\cap \{P_2>0\})}^p &=\|e^{it\Delta}f\|_{BL_{k,A}^p L^r (U\cap \{P_1<0\}\cap \{P_2<0\})}^p\,. \notag
\end{align}
Continuing inductively, we construct polynomials $P_1, P_2, \cdots, P_s$. Let $P:=\prod_{l=1}^s P_l$. The sign conditions of the polynomials cut $(\mathbb{R}^n\times\mathbb{R})\backslash Z(P)$ into $2^s$ cells $O_i$, and by construction and triangle inequality we have that for each $i$
\[
  \|e^{it\Delta}f\|_{BL_{k,A}^p L^r (U)}^p \leq 2^s \|e^{it\Delta}f\|_{BL_{k,A}^p L^r (U\cap O_i)}^p.
\]
 By construction, $\text{deg}\, P_l \lesssim 2^{(l-1)/(n+1)}$, therefore $\text{deg}\, P \leq c_n2^{s/(n+1)}$. We can choose $s$ such that $c_n2^{s/(n+1)} \in [D/2,D]$, then we have $\text{deg} \,P \leq D$ and the number of cells $2^s \sim_{n} D^{n+1}$.
\end{proof}

We write $Z(P_1,\cdots,P_{n+1-m})$ for the set of common zeros of the polynomials $P_1,\cdots,P_{n+1-m}$.
Throughout the paper, we will work with a nice class of varieties called transverse complete intersections. Here is the definition.
\begin{definition}
  We say that a variety $Z(P_1,\cdots, P_{n+1-m})$ is a transverse complete intersection if $\nabla P_1(z)\wedge\cdots \nabla P_{n+1-m}(z) \neq 0$ for each point $z$ in $Z(P_1,\cdots,P_{n+1-m})$.
\end{definition}
Guth (see Lemma 5.1. in \cite{lG16}) proved the following result using Sard's theorem, which guarantees that there are lots of transverse complete intersections.
\begin{lemma} \label{tci} \emph{[Guth]}\,
  If $P$ is a polynomial on $\ZR^n\times \ZR$, then for almost every $c_0\in \ZR$, $Z(P+c_0)$ is a transverse complete intersection.

  More generally, suppose that $Z(P_1,\cdots,P_{n+1-m})$ is a transverse complete intersection and that $P$ is another polynomial. Then for almost every $c_0\in \ZR$, $Z(P_1,\cdots,P_{n+1-m},P+c_0)$ is a transverse complete intersection.
\end{lemma}

The following partitioning theorem follows from the proof of Theorem \ref{Thm:PolynPart}. It is designed to allow small perturbations. In combination with Lemma \ref{tci}, it allows us to arrange that all the varieties that appear in our argument are transverse complete intersections.

\begin{theorem} \label{Thm:PolynPart2}
  Suppose that $f$ is a function with ${\rm supp}\widehat{f}\subset B(0,1)\subset \ZR^n$, $U$ is a subset of $B^n(0,R)\times [0,R]$, and $1\leq p, r<\infty$, then for each $D$ there exists a sequence of polynomials $Q_1,\cdots, Q_s$ on $\ZR^n\times \ZR$ with the following properties. We have $\sum_{l=1}^{s} {\rm Deg}\, Q_l\leq D$ and $2^s \sim D^{n+1}$. Let $P=\prod_{l=1}^s \tilde{Q}_l=\prod_{l=1}^s (Q_l+c_l)$ where $c_l \in \ZR$. Let $O_i$ be the open sets given by the sign conditions of $\tilde{Q}_l$. There are $2^s \sim D^{n+1}$ cells $O_i$ and $(\ZR^{n}\times \ZR )\backslash Z(P)=\cup O_i$.

  If the constants $c_l$ are sufficiently small, then for each $O_i$,
  \[
  \|e^{it\Delta}f\|_{BL_{k,A}^p L^r (U)}^p \leq c_n D^{n+1}\|e^{it\Delta}f\|_{BL_{k,A}^p L^r (U\cap O_i)}^p.
  \]
\end{theorem}

\section{Wave Packets Decomposition}\label{WPD}
\setcounter{equation}0

We focus on the dimension $n=2$ in the rest of the paper.

A (dyadic) rectangle in $\mathbb R^{2}$ is a product of (dyadic) intervals with respect to given coordinate axes of
$\mathbb R^{2}$. Two (dyadic) rectangles $\theta=\prod_{j=1}^{2}{\theta_j}$ and $\nu=\prod_{j=1}^{2}\nu_j$ are said to
be dual if $|\theta_j||\nu_j|=1$ for $j=1, 2$ and they share the same coordinate axes. We say that $(\theta ,\nu)$ is a tile if it is a pair of two  ${dual}$
(dyadic) rectangles.  The dyadic condition is not essential in our decomposition.

Let $\varphi$ be a Schwartz function from $\mathbb R $ to $\mathbb R$ for which $\widehat{\varphi}\geq 0$ is supported
in a small interval, of radius $\kappa$ ($\kappa$ is a fixed small constant), about the origin in $\mathbb R$, and it is identically
$1$ on another smaller interval around the origin. For a (dyadic) rectangular box $\theta=\prod_{j=1}^2 \theta_j$,
set
\begin{equation}
 \varphi_\theta(x_1, x_2) = \prod_{j=1}^{2} \frac{1}{|\theta_j|^{1/2}}\varphi\left( \frac{x_j-c(\theta_j)}{|\theta_j|}\right)\,.
\end{equation}
Here $c(\theta_j)$ is the center of the interval $\theta_j$ and hence $c(\theta)=\left(c(\theta_1),  c(\theta_{2})\right)$
 is the center of the rectangle $\theta$.
For a tile $(\theta, \nu)$ and $x\in \mathbb R^{2}$, we define
\begin{equation}
\varphi_{\theta, \nu}(x)= e^{2\pi ic(\nu)\cdot x}\varphi_\theta(x).
\end{equation}
We say that the dimensions of a tile $(\theta, \nu)$ are $\prod_{j=1}^{2} |\theta_j|$ for
$\theta= \prod_{j=1}^{2} \theta_j$.
Let $\bT$ be a collection of all tiles with fixed dimensions and coordinated axes.
Then for any Schwartz function $f$ from $\mathbb R^{2}$ to $ \mathbb R$, we have the following
representation
\begin{equation}\label{rep0}
 f(x) = c_\kappa \sum_{(\theta, \nu)\in \bT} f_{\theta,\nu}
 := c_\kappa \sum_{(\theta, \nu)\in \bT}\langle f, \varphi_{\theta,\nu}\rangle \varphi_{\theta,\nu}(x)\,,
\end{equation}
where $c_\kappa$ is an absolute constant.
This representation can be proved directly (see \cite{LL}) or
by employing inductively the one-dimensional result in \cite{LT1}. \\

We only need to focus on those tiles that can be written as a pair of
$R^{-\frac12}$-cube and $R^{\frac12}$-cube in $\mathbb R^2$,
when we apply the wave packets decomposition for a function with Fourier support in
a given $R^{-\frac12}$-cube.  Indeed, let $\theta$ be a $R^{-\frac12}$-cube (or ball)
in $B(0,1)\subset \ZR^2$. Let $\bT_\theta$ be a collection of all tiles $(\theta',\nu)$
such that $\nu$'s are $R^{\frac12}$-cubes and $\theta'=\theta$. Then for
Schwartz function $f$ with ${\rm supp}\wh{f}\subset B(0,1)$, we have
\begin{equation}\label{rep1}
 f(x) = c_\kappa\sum_{\theta}\sum_{(\theta',\nu)\in\bT_\theta}\langle f, \varphi_{\theta',\nu}\rangle \varphi_{\theta',\nu}(x)\,.
\end{equation}
Here $\theta$'s range over all possible cubes in ${\rm supp}\widehat{f}$. We use
$\bT$ to denote $\cup_\theta \bT_\theta$. It is clear that
\begin{equation}
 \sum_{(\theta,\nu)\in \bT}\big|\langle f,\varphi_{\theta,\nu}\rangle\big|^2\sim \|f\|_2^2\,.
\end{equation}
We set
\begin{equation}
 \psi_{\theta,\nu}(x, t) = e^{it\Delta}\varphi_{\theta,\nu}(x)\,.
\end{equation}
From (\ref{rep1}), we end up with the following representation for $e^{it\Delta}f$ :
\begin{equation}\label{rep2}
 e^{it\Delta}f(x) =c_\kappa \sum_{(\theta,\nu)\in \bT} e^{it\Delta}f_{\theta,\nu}(x)= c_\kappa \sum_{(\theta,\nu)\in \bT} \langle f, \varphi_{\theta,\nu}\rangle\psi_{\theta,\nu}(x, t)\,.
\end{equation}
We shall analyze the localization of $\psi_{\theta,\nu}$ in the time-frequence space.
Notice that $t$ is restricted to $[0, R]$. Let $\psi$ be a Schwartz function with
$\wh{\psi}$ supported in $[-1, 1]$ and $2\psi(t)\geq \chi_{[0,1]}(t)$.
Here $\chi_{[0,1]}$ is the characteristic function on $[0,1]$.
We can replace $\psi_{\theta,\nu}$ by
$\psi_{\theta,\nu}^*$ in (\ref{rep2}), where
\begin{equation}
 \psi_{\theta,\nu}^*(x,t) =\psi_{\theta,\nu}(x, t) \psi\big(\frac{t}{R}\big)\,.
\end{equation}
Let
\begin{equation}\label{defofTs}
  T_{\theta,\nu} := \{(x,t)\in\mathbb{R}^2\times\mathbb{R}\,:\,0\leq t\leq R, |x-c(\nu)+2tc(\theta)|\leq R^{1/2+\delta}\}\, ,
\end{equation}
where $\delta$ is a small positive parameter. $T_{\theta,\nu}$ is a tube of length $R$, radius $R^{1/2+\delta}$, with the
direction $G_0(\theta)=(-2c(\theta),1)$, and intersecting $\{t=0\}$ at a $R^{1/2+\delta}$-ball centered at $c(\nu)$.
From the definitions of $e^{it\Delta}$ and $\psi_{\theta,\nu}$, it is easy to see that, by integration by parts,
$\psi^*_{\theta,\nu}$ is almost supported in the tube $T_{\theta,\nu}$.
More precisely, we have
\begin{equation}\label{eq:psis}
 |\psi^*_{\theta,\nu}(x,t)|\leq \frac{1}{\sqrt{R}} \chi^*_{T_{\theta,\nu}}(x,t) \, ,
\end{equation}
where
$\chi^*_{T_{\theta,\nu}}$ denotes a bump function satisfying that
$\chi^*_{T_{\theta,\nu}}=1$ on $\{(x,t)\in\mathbb{R}^2\times\mathbb{R}\,:\,0\leq t\leq R, |x-c(\nu)+2tc(\theta)|\leq \sqrt{R}\}$,
and $\chi^*_{T_{\theta,\nu}}=O(R^{-1000})$ outside $T_{\theta,\nu}$.  $\chi^*_{T_{\theta,\nu}}$ essentially can be viewed as
$\chi_{T_{\theta,\nu}}$, the indicator function on the tube $T_{\theta,\nu}$.

On the other hand, the Fourier transform of $\psi_{\theta,\nu}^*$ enjoys
\begin{equation}
  \widehat{\psi^*_{\theta,\nu}}(\xi_1,\xi_2,\xi_3)
 =R\widehat{\varphi_{\theta,\nu}}(\xi_1,\xi_2)\widehat{\psi}\left(\frac{\xi_3-(\xi_1^2+\xi_2^2)}{1/R}\right)\,.
\end{equation}
Hence $\widehat{\psi^*_{\theta,\nu}}$ is supported in a $\frac{1}{R}$-neighborhood of parabolic cap over $\theta$, i.e.,
\begin{equation} \label{eq:psishat}
  \text{supp}\,\widehat{\psi^*_{\theta,\nu}} \subseteq \big\{(\xi_1,\xi_2,\xi_3)\,:\,
(\xi_1,\xi_2)\in \theta, |\xi_3-(\xi_1^2+\xi_2^2)|\leq \frac{1}{R} \big\}.
\end{equation}
We denote this $\frac{1}{R}$-neighborhood of parabolic cap over $\theta$ by $\theta^*$.
In the rest of the paper, we can assume that the function $\psi_{\theta,\nu}$ is essentially localized in
$T_{\theta,\nu}\times \theta^*$ in the time-frequence space.

\section{Main Proposition}
\label{MP}
\setcounter{equation}0
Now we set up the inductive argument to prove Theorem \ref{klq}. Let $m=1$ or $2$ denote the dimension of a variety. We choose small parameters $0<\delta \ll \delta_2 \ll \delta_1 \ll \delta_0 \ll \epsilon.$ We write ${\rm RapDec}(R)$ for terms rapidly decaying in $R$ which are negligible in our estimates. For a variety $Z$ and a point $z\in Z$, let $T_zZ$ denote the tangent space to Z at the point z.
We use the following definition from \cite{lG16} of a wave packet being tangent to a transverse complete intersection $Z$.
\begin{definition}
  Suppose that $Z=Z(P_1,\cdots,P_{3-m})$ is a transverse complete intersection in $\ZR^2\times \ZR$. We say that $T_{\theta, \nu}$ is $R^{-\frac 1 2 +\delta_m}$-tangent to Z in $B^*_R$ if the following two conditions hold:

\begin{itemize}
\item Distance condition:
\[
  T_{\theta,\nu}\subset N_{R^{\frac 1 2 +\delta_m}}(Z) \cap B^*_R.
\]
\item Angle condition: If $z\in Z\cap B^*_R\cap N_{O(R^{\frac 1 2 +\delta_m})}T_{\theta,\nu} $, then
\[
{\rm Angle}(G(\theta),T_z Z) \lesssim R^{-1/2+\delta_m}.
\]
\end{itemize}

We define
\[
\bT_Z:=\{(\theta,\nu)\,|\,T_{\theta,\nu} \,{\rm is} \,R^{-\frac 1 2 +\delta_m }\text{-tangent to}\, Z\}.
\]

We say that $f$ is concentrated in wave packets from $\bT_Z$ if
\[
 \sum_{(\theta,\nu)\notin \bT_Z} \|f_{\theta,\nu}\|_2 \leq {\rm RapDec}(R)\|f\|_2.
\]
\end{definition}
All functions $f$ that appear in the following context satisfy the assumptions in Theorem \ref{klq}, i.e. ${\rm supp}\widehat{f}\subset B(\xi_0,M^{-1})$ with arbitrary $\xi_0 \in B(0,1)$ and $M\geq 1$.
\begin{proposition} \label{goal}
Fix $k=2$. For $\epsilon>0$, there are small constants $0<\delta\ll \delta_2\ll \delta_1 \ll \delta_0 \ll \epsilon$, and a large constant $\bar{A}=\bar{A}(\epsilon)$ so that the following holds for any $q>\delta^{-1}$ :

(a) Suppose that $Z=Z(P_1,P_2)\subset \ZR^2\times \ZR$ is a transverse complete intersection. Suppose that $f$ is concentrated in wave packets from $T_Z$. Then for any $1\leq A\leq \bar A$, any radius $R\geq 1$, and any $p\geq 2$,
\begin{equation}\label{a}
  \|e^{it\Delta}f\|_{BL_{k,A}^p L^q (B^*_R)}\leq
  \begin{cases}
    R^{\frac{3-p}{2p}+\epsilon}\|f\|_2, &{\rm if}\, KM \geq R^{\frac 1 2 -O(\delta_0)}, \\
    {\rm RapDec}(R)\|f\|_2, &{\rm otherwise}.
  \end{cases}
\end{equation}

(b) Suppose that $Z=Z(P) \subset \ZR^2\times \ZR$ is a transverse complete intersection, where ${\rm Deg}P\leq D_Z$. Suppose that $f$ is concentrated in wave packets from $T_Z$. Then for any $1\leq A\leq \bar{A} $, any radius $R\geq 1$, and any $p>4$,
\begin{equation}\label{b}
  \|e^{it\Delta}f\|_{BL_{k,A}^p L^q (B^*_R)}\leq C(\epsilon,K,D_Z) R^{\epsilon}R^{\delta(\log \bar A -\log A)}R^{\frac{1}{2p}-\frac{1}{4}}\|f\|_2.
\end{equation}

(c) For any $1\leq A\leq \bar{A} $, any radius $R\geq 1$, and any $1\leq p \leq 3.2$,
\begin{equation} \label{c}
  \|e^{it\Delta}f\|_{BL_{k,A}^p L^q (B^*_R)}\leq C(\epsilon,K) R^{\epsilon}R^{\delta(\log \bar A -\log A)}R^{\frac{2}{p}-\frac{5}{8}}\|f\|_2.
\end{equation}
\end{proposition}

\begin{remark}
Theorem \ref{klq} follows immediately from part (c) of Proposition \ref{goal} by taking $A=\bar A$. The remaining part of the paper is devoted to a proof of Proposition \ref{goal}.
\end{remark}

In the rest of this section, we prove part (a) of Proposition \ref{goal}. The dimension of $Z$ is $m=1$.

Note that in the case $KM=\sqrt{R}$, for each $f_\tau$ all associated wave packets are in the same direction, then by a direct computation we have that
\begin{align}
   \|e^{it\Delta}f\|_{BL_{k,A}^p L^q (B^*_R)}^p &\leq \sum_\tau \sum_{B_K\subset B_R}\left[\sum_{I_K\subset [0,R]} \left(\int_{B_K\times I_K}|e^{it\Delta}f_\tau|^p\right)^q\right]^{\frac 1 q} \notag \\
  & \lesssim_K  R^{\frac 1 q} \sum_\tau \int_{B_R} \underset{t\in [0,R]}{{\rm sup}} |e^{it\Delta}f_\tau|^p \lesssim R^{\frac{3-p}{2}+O(\delta)}\|f\|_2^p. \notag
\end{align}
where the last inequality follows from wave packets decomposition, the bound \eqref{eq:psis} and the fact that tubes arising from wave packets with the same direction are essentially disjoint. This proves part (a) in the case $KM\geq R^{\frac 1 2 -O(\delta_0)}$.

Now suppose that $KM< R^{\frac 1 2 -O(\delta_0)}$. For each $B_K\times I_K$ which intersects $N_{R^{\frac 1 2 + \delta_1}}(Z)$ in $B_R^*$, we pick a point $z_0\in Z\cap N_{R^{\frac 12 +\delta_1}}(B_K\times I_K)$. For each $(\theta, \nu)\in \bT_Z$ with $T_{\theta,\nu}\cap (B_K\times I_K)\neq \emptyset$, we have that $z_0 \in Z\cap B_R^* \cap N_{O(R^{\frac 12 +\delta_1})}T_{\theta,\nu}$. Then by definition of $\bT_Z$, we have
$${\rm Angle}(G(\theta), T_{z_0}Z)\lesssim R^{-\frac 12 +\delta_1}.$$ Then for any $\tau$ with such a $\theta$ in it, we have
$${\rm Angle}(G(\tau), T_{z_0}Z)\leq (KM)^{-1}.$$ Since $T_{z_0}Z$ is a $1$-subspace and $A\geq 1$, by definition \eqref{mu} such balls $\tau$ do not contribute to $\mu_{e^{it\Delta}f}(B_K\times I_K)$. Since $f$ is concentrated in wave packets from $T_Z$, this completes the proof of part (a).

\section {Proof of Proposition \ref{goal} Part (b)}

We prove part (b) by induction. The dimension of $Z$ is $m=2$. Note that when $R$ is small, we choose the constant $C(\epsilon, K,D_Z)$ sufficiently large and the result follows. So we can assume that $R$ is large compared to $\epsilon$, $K$ and $D_Z$. For the case $A=1$, we choose $\bar A$ large enough so that $R^{\delta\log\bar A}=R^{100}$ and the result follows. So we can also induct on $A$.

We can assume that $KM< R^{\frac 12 -O(\delta_0)}$. If $KM\geq R^{\frac 12 -O(\delta_0)} $,  the same direct computation as in the proof of \eqref{a} gives us a bound $R^{\frac{3-p}{2p}+\epsilon}\|f\|_2$, which is better than $R^{\frac{1}{2p}-\frac 1 4+\epsilon}\|f\|_2$ when $p>4$.

We let $D=D(\epsilon,D_Z)$ be a function that we will define later. We say that we are in \textbf{algebraic case} if there is a transverse complete intersection $Y^1 \subset Z^2$ of dimension $1$, defined using polynomials of degree $\leq D(\epsilon,D_Z)$, so that
\[
  \|e^{it\Delta}f\|_{BL_{k,A}^p L^q (B^*_R)}\lesssim \|e^{it\Delta}f\|_{BL_{k,A}^p L^q (B^*_R\cap N_{R^{1/2+\delta_2}}(Y))}.
\]
Otherwise we say that we are in \textbf{cellular case}.

\subsection {Cellular case} In cellular case, we will use polynomial partitioning. In the same way as Guth did in \cite{lG16} (Section 8), we first identify a significant piece $N_1$ of $B^*_R\cap N_{R^{\frac 12 +\delta_2}}Z$, where locally $Z$ behaves like a $2$-plane $V$, next apply the polynomial partitioning result -- Theorem \ref{Thm:PolynPart2} to the push-forward of $\|e^{it\Delta}f\|_{BL_{k,A}^p L^q (N_1)}^p$ on $V$, then pull the polynomial on $V$ back to a polynomial on $\ZR^2\times \ZR$, via the orthogonal projection $\pi: \ZR^2\times \ZR \rightarrow V$.
We have the following partitioning result in cellular case: there exists a non-zero polynomial $Q$ of degree at most D such that $(\mathbb{R}^2\times\mathbb{R})\backslash Z(Q)$ is a union of $\sim D^{2}$ disjoint open sets $O_i$ and for $\sim D^2$ cells $O_i$ we have
  \begin{equation}\label{O_i}
  \|e^{it\Delta}f\|_{BL_{k,A}^p L^q (B^*_R)}^p \lesssim  D^{2}\|e^{it\Delta}f\|_{BL_{k,A}^p L^q (B^*_R\cap O'_i)}^p.
  \end{equation}
 where $O'_i:=O_i\backslash W$ and $W:=N_{R^{1/2+\delta}}Z(Q)$.

For each cell $O'_i$, we set
\[
 \bT_i:=\{(\theta,\nu)\in\bT\,:\,T_{\theta,\nu} \cap O'_i \neq \emptyset\} \,.
\]
Given the function $f$, we define
\[
  f_i =  \sum_{(\theta,\nu)\in \bT_i} f_{\theta,\nu} \,.
\]
From \eqref{eq:psis}, it follows that on $O_i'$,
\[
 e^{it\Delta}f(x)\sim e^{it\Delta}f_i(x) \,.
\]
Due to the fundamental theorem of Algebra, we have a simple but important geometric observation:
\begin{lemma}\label{lem-G0}
  $\#\{i\,:\,(\theta,\nu) \in \bT_i\}\leq D+1$ for any $(\theta,\nu)\in\bT$.
\end{lemma}
\begin{proof}
 If $(\theta,\nu)\in \bT_i$, then the central line of $T_{\theta,\nu}$ must cross $O_i$. On the other hand, a line can intersect $Z(Q)$
at most $D$ times, hence can cross at most $D+1$ cells $O_i$.
\end{proof}

By Lemma \ref{lem-G0},
\[
  \sum_{i} \|f_{i}\|_2^2
\lesssim (D+1) \sum_{\theta,\nu}\| f_{\theta,\nu}\|_2^2 \lesssim D\|f\|_2^2\,.
\]
Henceforth, by pigeonhole principle, there exists $O'_i$ satisfying \eqref{O_i} such that
\[
  \|f_i\|_2^2 \lesssim D^{-1}\|f\|_2^2.
\]
Now we apply \eqref{b} to this special $f_i$ at radius $\frac R 2$,
\begin{align}
  &\|e^{it\Delta}f\|_{BL_{k,A}^p L^q (B^*_R)}^p \lesssim D^{2}\|e^{it\Delta}f\|_{BL_{k,A}^p L^q (B^*_R\cap O'_i)}^p \sim  D^{2}\|e^{it\Delta}f_i\|_{BL_{k,A}^p L^q (B^*_R)}^p \notag \\
   \lesssim & D^2 \left[C(\epsilon,K,D_Z) R^{\epsilon}R^{\delta(\log \bar A -\log A)}R^{\frac{1}{2p}-\frac{1}{4}}\|f_i\|_2\right]^p \notag \\
   \lesssim & D^{2-\frac p 2} \left[C(\epsilon,K,D_Z) R^{\epsilon}R^{\delta(\log \bar A -\log A)}R^{\frac{1}{2p}-\frac{1}{4}}\|f\|_2\right]^p \notag
\end{align}

We choose $D$ large enough so that, for $p>4$ we have $D^{2-\frac p 2} \ll 1$. Therefore, the cellular case is done by induction.

\subsection{Algebraic tangential case}

In algebraic case, we pick $Y^1\subset Z^2$ of dimension $1$, defined using polynomials of degree $\leq D=D(\epsilon, D_Z)$, so that
\[
  \|e^{it\Delta}f\|_{BL_{k,A}^p L^q (B^*_R)}\lesssim \|e^{it\Delta}f\|_{BL_{k,A}^p L^q (B^*_R\cap N_{R^{1/2+\delta_2}}(Y))}.
\]
Then we decompose $B^*_R$ into smaller balls $B_j$ of radius $\rho$, where $\rho^{\frac 12+\delta_1}=R^{\frac 1 2 +\delta_2}$. Recall that $\delta_2\ll \delta_1$, so $\rho \sim R^{1-O(\delta_1)}$. For each $j$, we define $f_j:=\sum_{(\theta,\nu)\in \bT_j}f_{\theta,\nu}$, where
\[
  \bT_j:= \left\{ (\theta,\nu) \in \bT_Z\,|\,T_{\theta,\nu}\cap N_{R^{\frac 12 +\delta_2}}(Y)\cap B_j \neq \emptyset \right\}.
\]
On $B_j $, $e^{it\Delta}f_j \sim e^{it\Delta}f$. Therefore
\[
  \|e^{it\Delta}f\|_{BL_{k,A}^p L^q (B^*_R)}^p\lesssim \sum_{j}\|e^{it\Delta}f_j\|_{BL_{k,A}^p L^q (B_j)}^p\,.
\]
In order to induct on the dimension of the variety, we further divide $\bT_j$ into tubes that are tangential to $Y$ and tubes that are transverse to $Y$. We say that $T_{\theta,\nu}$ is \textbf{tangential} to $Y$ in $B_j$ if the following two conditions hold:
\begin{itemize}
\item Distance condition:
\[
  T_{\theta,\nu}\cap 2B_j \subset N_{R^{\frac 1 2 +\delta_2}}(Y) \cap 2 B_j = N_{\rho^{\frac 1 2 +\delta_1}}(Y) \cap 2 B_j.
\]
\item Angle condition:

If $y\in Y\cap 2B_j \cap N_{O(R^{\frac 1 2 +\delta_2})}T_{\theta,\nu} =Y\cap 2B_j \cap N_{O(\rho^{\frac 1 2 +\delta_1})}T_{\theta,\nu}$, then
\[
{\rm Angle}(G(\theta),T_y Y) \lesssim \rho^{-1/2+\delta_1}.
\]
\end{itemize}
We define the tangential wave packets by
\[
  \bT_{j,{\rm tang}}:=\left\{(\theta,\nu)\in \bT_j\,|\,T_{\theta,\nu} \,\text{is tangent to }\, Y \,{\rm in}\, B_j\right\}.
\]
And we define the transverse wave packets by
\[
  \bT_{j,{\rm trans}}:=\bT_j \backslash \bT_{j,{\rm tang}}.
\]
We define $f_{j,{\rm tang}}:=\sum_{(\theta,\nu)\in \bT_{j,{\rm tang}}}f_{\theta,\nu}$ and $f_{j,{\rm trans}}:=\sum_{(\theta,\nu)\in \bT_{j,{\rm trans}}}f_{\theta,\nu}$, so
\[
  f_j=f_{j,{\rm tang}}+f_{j,{\rm trans}}.
\]
Therefore we bound $\|e^{it\Delta}f\|_{BL_{k,A}^p L^q (B^*_R)}^p$ by
\[
   \sum_{j}\|e^{it\Delta}f_{j,{\rm tang}}\|_{BL_{k,A/2}^p L^q (B_j)}^p+\sum_{j}\|e^{it\Delta}f_{j,{\rm trans}}\|_{BL_{k,A/2}^p L^q (B_j)}^p\,.
\]
We will bound the tangential term by induction on the dimension, and bound the transverse term by induction on the radius $R$. In order to apply induction on the ball $B_j$, we need to redo the wave packets decomposition at a scale $\rho$ instead of $R$. See Section 7 in \cite{lG16} for details.

First suppose that the tangential term dominates. By the definition of $\bT_{j,{\rm tang}}$, the new wave packets for $f_{j,tang}$ are $\rho^{-\frac 12 +\delta_1}$-tangent to $Y$ in $B_j$, so we can apply \eqref{a} to $f_{j,{\rm tang}}$:
\[
 \sum_{j}\|e^{it\Delta}f_{j,{\rm tang}}\|_{BL_{k,A/2}^p L^q (B_j)}^p \leq {\rm RapDec} (R) \|f\|_2^p,
\]
note that $KM\leq R^{\frac 12 -O(\delta_0)}=\rho^{(1-O(\delta_1))(\frac 12 -O(\delta_0))}=\rho^{\frac 12 -O(\delta_0)}$. So the induction on algebraic tangential term closes.

\subsection{Algebraic transverse case}
In this subsection, we estimate $$\sum_{j}\|e^{it\Delta}f_{j,{\rm trans}}\|_{BL_{k,A/2}^p L^q (B_j)}^p$$ by induction on the radius $R$.

We will use the following geometric lemma from \cite{lG16}, which is about how a tube intersects a variety transversely.
\begin{lemma}\label{transverse} \emph{[Guth]}\,\,
Suppose that $T$ is a tube of radius $r$ and with direction $\upsilon(T)$. Suppose that $Z=Z(P_1,\cdots,P_{n+1-m})\subset \ZR^n\times\ZR$ is a transverse complete intersection defined by polynomials of degree at most $D$. Define $$Z_{>\alpha}:=\{z\in Z\,|\,{\rm Angle}(\upsilon(T),T_zZ)>\alpha\}.$$ Then for any $\alpha>0$, $Z_{>\alpha}\cap T$ is contained in $\lesssim D^{n+1}$ balls of radius $\lesssim r\alpha^{-1}$.
\end{lemma}

For each $f_{j,{\rm trans}}$ we consider the associated new wave packets $(\tilde{\theta},\tilde{\nu})$ at scale $\rho$. The new tubes $T_{\tilde{\theta},\tilde{\nu}}$ are of radius $\rho^{\frac 12 +\delta}$ and of length $\rho$. The new tubes are no longer $\rho^{-\frac 1 2 +\delta_2}$-tangent to $Z$ in $B_j$, they satisfy the angle condition but not the distance condition, more precisely, the new tubes are contained in $N_{R^{\frac 1 2 +\delta_2}}(Z)\cap B_j$, but not necessarily contained in $N_{\rho^{\frac 1 2 +\delta_2}}(Z)\cap B_j$. So we cover $N_{R^{\frac 1 2 +\delta_2}}(Z)\cap B_j$ with disjoint translates of $N_{\rho^{\frac 1 2 +\delta_2}}(Z)\cap B_j$. By the angle condition, it turns out that each new tube lies in one of these translates. For any $b\in B^*_{R^{\frac 12 +\delta_m}}$, define
\[
  \tilde{\bT}_{Z+b}:=\{(\tilde{\theta},\tilde{\nu})\,|\, T_{\tilde{\theta},\tilde{\nu}} \, \text{is}\, \rho^{-\frac 1 2 +\delta_2}\text{-tangent to}\, Z+b \,\text{in}\, B_j\}.
\]
We choose a random set of vectors $b\in B^*_{R^{\frac12 +\delta_2}}$, and using the new wave packets from $\tilde{\bT}_{Z+b}$ we get functions $f_{j,{\rm trans},b}$ satisfying the following properties (see Section 7 and Section 8 in \cite{lG16}):

\begin{itemize}
 \item $|e^{it\Delta}f_{j,{\rm trans},b}(x)|\sim \chi_{N_{\rho^{\frac 1 2 +\delta_2}}(Z+b)}(x,t)|e^{it\Delta}f_{j,{\rm trans}}(x)|\,,$

 \item $\|e^{it\Delta}f_{j,{\rm trans}}\|_{BL_{k,A/2}^p L^q (B_j)}^p \lesssim (\log R) \sum_b \|e^{it\Delta}f_{j,{\rm trans},b}\|_{BL_{k,A/2}^p L^q (B_j)}^p \,,$

 \item  $\sum_j \sum_b \|f_{j,{\rm trans},b}\|_2^2 \lesssim \sum_j \|f_{j,{\rm trans}}\|_2^2 \lesssim_D \|f\|_2^2\,,$ where the second inequality follows from Lemma \ref{transverse}.

 \item $\underset{b}{{\rm max}} \|f_{j,{\rm trans},b}\|_2^2 \leq R^{O(\delta_2)}\left(\frac{R^{1/2}}{\rho^{1/2}}\right)^{-1}\|f_{j,{\rm trans}}\|_2^2 \,.$

\end{itemize}
Now we have
\[
  \sum_j \|e^{it\Delta}f_{j,{\rm trans}}\|_{BL_{k,A/2}^p L^q (B_j)}^p \lesssim (\log R) \sum_j \sum_b \|e^{it\Delta}f_{j,{\rm trans},b}\|_{BL_{k,A/2}^p L^q (B_j)}^p\,.
\]
We use \eqref{b} to bound $\|e^{it\Delta}f_{j,{\rm trans},b}\|_{BL_{k,A/2}^p L^q (B_j)}$ by
\[
  C(\epsilon,K,D_Z) \rho^{\epsilon}\rho^{\delta(\log \bar A -\log \frac A 2)}\rho^{\frac{1}{2p}-\frac{1}{4}}\|f_{j,{\rm trans },b}\|_2\,.
\]
We write $\sum_j\sum_b \|f_{j,{\rm trans },b}\|_2^p \leq \sum_j\sum_b \|f_{j,{\rm trans },b}\|_2^2 \cdot \underset{b}{{\rm max}} \|f_{j,{\rm trans},b}\|_2^{p-2}$, then using the above properties we get
\begin{align}
  &\sum_j \|e^{it\Delta}f_{j,{\rm trans}}\|_{BL_{k,A/2}^p L^q (B_j)}^p \notag\\
  \lesssim & R^{O(\delta_2)}\left(\frac R\rho\right)^{-\epsilon p} \left[ C(\epsilon,K,D_Z) R^{\epsilon}R^{\delta(\log \bar A -\log A)}R^{\frac{1}{2p}-\frac{1}{4}}\|f\|_2 \right]^p \,.\notag
\end{align}
Since $\frac R \rho=R^{O(\delta_1)}$, by choosing $\delta_2 \ll \epsilon \delta_1$ the induction closes for the algebraic transverse term. And this completes the proof of part (b) of Proposition \ref{goal}.

\section{Proof of Proposition \ref{goal} Part (c)}
To prove part (c), we only need to focus on the endpoint $p=3.2$. Once we prove part (c) for $p=3.2$,  then the whole range in part (c) will follow from H\"older's inequality \eqref{holder}. We fix $p=3.2$.

The proof of part (c) is similar to the proof of part (b). Again we prove part (c) by induction on the dimension, the radius $R$ and on $A$.

We can assume that $KM\leq R^{\frac 12 -O(\delta_0)}$. If $KM\geq R^{\frac 12 -O(\delta_0)} $,  the same direct computation as in the proof of \eqref{a} can give a bound $R^{\frac{3-p}{2p}+\epsilon}\|f\|_2$, which is better than $R^{\frac{2}{p}-\frac 5 8+\epsilon}\|f\|_2$ when $p<4$.

We let $D=D(\epsilon)$ be a function that we will define later. We say that we are in \textbf{algebraic case} if there is a transverse complete intersection $Z$ of dimension $2$, defined using a polynomial of degree $\leq D$, so that
\[
  \|e^{it\Delta}f\|_{BL_{k,A}^p L^q (B^*_R)}\lesssim \|e^{it\Delta}f\|_{BL_{k,A}^p L^q (B^*_R\cap N_{R^{1/2+\delta}}(Z))}.
\]
Otherwise we say that we are in \textbf{cellular case}.

\subsection {Cellular case} In cellular case, we will use the polynomial partitioning result. By Theorem \ref{Thm:PolynPart2}, there exists a non-zero polynomial $P=\prod_l Q_l$ of degree at most D such that $(\mathbb{R}^2\times\mathbb{R})\backslash Z(P)$ is a union of $\sim D^{3}$ disjoint open sets $O_i$ and for each cell $O_i$ we have
  \begin{equation}
  \|e^{it\Delta}f\|_{BL_{k,A}^p L^q (B^*_R)}^p \lesssim  D^{3}\|e^{it\Delta}f\|_{BL_{k,A}^p L^q (B^*_R\cap O_i)}^p.
  \end{equation}
Moreover, by Lemma \ref{tci} we can guarantee that for each $l$, $Z(Q_l)$ is a transverse complete intersection.

Next we define
\[
W:=N_{R^{1/2+\delta}}Z(P),\,\, O'_i:=O_i\backslash W\,.
\]
Since $W\subset \bigcup_l N_{R^{1/2+\delta}}Z(Q_l)$ and we are in cellular case, the contribution from $W$ is negligible. Hence for $\sim D^3$ cells $O'_i$, we have
\begin{equation}\label{O_i2}
  \|e^{it\Delta}f\|_{BL_{k,A}^p L^q (B^*_R)}^p \lesssim  D^{3}\|e^{it\Delta}f\|_{BL_{k,A}^p L^q (B^*_R\cap O'_i)}^p.
\end{equation}

For each cell $O'_i$, we set
\[
 \bT_i=\{(\theta,\nu)\in\bT\,:\,T_{\theta,\nu} \cap O'_i \neq \emptyset\} \,.
\]
Given the function $f$, we define
\[
  f_i :=  \sum_{(\theta,\nu)\in \bT_i} f_{\theta,\nu} \,.
\]
From \eqref{eq:psis}, it follows that on $O_i'$,
\[
 e^{it\Delta}f(x)\sim e^{it\Delta}f_i(x) \,.
\]
Again by the fundamental theorem of Algebra, we have
\[
 \#\{i\,:\,(\theta,\nu) \in \bT_i\}\leq D+1 ,\,\,\text{for any}\, (\theta,\nu)\in\bT.
\]
Hence
\[
  \sum_{i} \|f_{i}\|_2^2
\lesssim (D+1) \sum_{\theta,\nu}\| f_{\theta,\nu}\|_2^2 \lesssim D\|f\|_2^2\,.
\]
Henceforth, by pigeonhole principle, there exists $O'_i$ satisfying \eqref{O_i2} such that
\[
  \|f_i\|_2^2 \lesssim D^{-2}\|f\|_2^2.
\]
Now we apply \eqref{c} to this special $f_i$ at radius $\frac R 2$,
\begin{align}
  &\|e^{it\Delta}f\|_{BL_{k,A}^p L^q (B^*_R)}^p \lesssim D^{3}\|e^{it\Delta}f\|_{BL_{k,A}^p L^q (B^*_R\cap O'_i)}^p \sim  D^{3}\|e^{it\Delta}f_i\|_{BL_{k,A}^p L^q (B^*_R)}^p \notag \\
   \lesssim & D^3 \left[C(\epsilon,K) R^{\epsilon}R^{\delta(\log \bar A -\log A)}R^{\frac{2}{p}-\frac{5}{8}}\|f_i\|_2\right]^p \notag \\
   \lesssim & D^{3-p} \left[C(\epsilon,K) R^{\epsilon}R^{\delta(\log \bar A -\log A)}R^{\frac{2}{p}-\frac{5}{8}}\|f\|_2\right]^p \notag
\end{align}

We choose $D$ large enough so that, for $p>3$ we have $D^{3-p} \ll 1$. Then the cellular case is done by induction.

\subsection{Algebraic tangential case}

In algebraic case, we pick a transverse complete intersection $Z$ of dimension $2$, defined using a polynomial of degree $\leq D=D(\epsilon)$, so that
\[
  \|e^{it\Delta}f\|_{BL_{k,A}^p L^q (B^*_R)}\lesssim \|e^{it\Delta}f\|_{BL_{k,A}^p L^q (B^*_R\cap N_{R^{1/2+\delta}}(Z))}.
\]
Then we decompose $B^*_R$ into smaller balls $B_j$ of radius $\rho$, where $\rho^{\frac 12+\delta_2}=R^{\frac 1 2 +\delta}$. Recall that $\delta\ll \delta_2$, so $\rho \sim R^{1-O(\delta_2)}$. For each $j$, we define $f_j:=\sum_{(\theta,\nu)\in \bT_j}f_{\theta,\nu}$, where
\[
  \bT_j:= \left\{ (\theta,\nu) \in \bT_Z\,|\,T_{\theta,\nu}\cap N_{R^{\frac 12 +\delta}}(Z)\cap B_j \neq \emptyset \right\}.
\]
On $B_j $, $e^{it\Delta}f_j \sim e^{it\Delta}f$. Therefore
\[
  \|e^{it\Delta}f\|_{BL_{k,A}^p L^q (B^*_R)}^p\lesssim \sum_{j}\|e^{it\Delta}f_j\|_{BL_{k,A}^p L^q (B_j)}^p\,.
\]
In order induct on the dimension of the variety, we further divide $\bT_j$ into tubes that are tangential to $Z$ and tubes that are transverse to $Z$. We say that $T_{\theta,\nu}$ is \textbf{tangential} to $Z$ in $B_j$ if the following two conditions hold:
\begin{itemize}
\item Distance condition:
\[
  T_{\theta,\nu}\cap 2B_j \subset N_{R^{\frac 1 2 +\delta}}(Z) \cap 2 B_j = N_{\rho^{\frac 1 2 +\delta_2}}(Z) \cap 2 B_j.
\]
\item Angle condition:

If $z\in Z\cap 2B_j \cap N_{O(R^{\frac 1 2 +\delta})}T_{\theta,\nu} =Z\cap 2B_j \cap N_{O(\rho^{\frac 1 2 +\delta_2})}T_{\theta,\nu}$, then
\[
{\rm Angle}(G(\theta),T_z Z) \lesssim \rho^{-1/2+\delta_2}.
\]
\end{itemize}
We define the tangential wave packets by
\[
  \bT_{j,{\rm tang}}:=\left\{(\theta,\nu)\in \bT_j\,|\,T_{\theta,\nu} \,\text{is tangent to }\, Z \,{\rm in}\, B_j\right\}.
\]
And we define the transverse wave packets by
\[
  \bT_{j,{\rm trans}}:=\bT_j \backslash \bT_{j,{\rm tang}}.
\]
We define $f_{j,{\rm tang}}:=\sum_{(\theta,\nu)\in \bT_{j,{\rm tang}}}f_{\theta,\nu}$ and $f_{j,{\rm trans}}:=\sum_{(\theta,\nu)\in \bT_{j,{\rm trans}}}f_{\theta,\nu}$, so
\[
  f_j=f_{j,{\rm tang}}+f_{j,{\rm trans}}.
\]
Therefore we bound $\|e^{it\Delta}f\|_{BL_{k,A}^p L^q (B^*_R)}^p$ by
\[
   \sum_{j}\|e^{it\Delta}f_{j,{\rm tang}}\|_{BL_{k,A/2}^p L^q (B_j)}^p+\sum_{j}\|e^{it\Delta}f_{j,{\rm trans}}\|_{BL_{k,A/2}^p L^q (B_j)}^p\,.
\]
Again we will bound the tangential term by induction on the dimension, and bound the transverse term by induction on the radius $R$. In order to apply induction on the ball $B_j$, we also need to redo the wave packets decomposition at a scale $\rho$ instead of $R$.

First suppose that the tangential term dominates. By the definition of $\bT_{j,{\rm tang}}$, the new wave packets for $f_{j,tang}$ are $\rho^{-\frac 12 +\delta_2}$-tangent to $Z$ in $B_j$, so we can apply \eqref{b} to $f_{j,{\rm tang}}$:
\[
 \|e^{it\Delta}f_{j,{\rm tang}}\|_{BL_{k,A/2}^r L^q (B_j)} \leq C(\epsilon/2,K,D) \rho^{\epsilon/2}\rho^{\delta(\log \bar A -\log (A/2))}\rho^{\frac{1}{2r}-\frac{1}{4}}\|f\|_2,
\]
for $r>4$. For $p= 3.2$, by H\"lder's inequality \eqref{holder} and by taking $r=4+\delta$ we have
\[
  \|e^{it\Delta}f_{j,{\rm tang}}\|_{BL_{k,A/2}^p L^q (B_j)} \leq C(\epsilon/2,K,D) \rho^{\epsilon/2}\rho^{\delta(\log \bar A -\log (A/2))}\rho^{O(\delta)}\|f\|_2
\]
Hence we get
\begin{align}
   &\sum_{j}\|e^{it\Delta}f_{j,{\rm tang}}\|_{BL_{k,A/2}^p L^q (B_j)}^p \notag \\
   \lesssim & \left[R^{O(\delta_2)}C(\epsilon/2,K,D) R^{\epsilon/2}R^{\delta(\log\bar A-\log A)}\|f\|_2\right]^p\,.
\end{align}
Since $R^{O(\delta_2)}R^{\epsilon/2} \leq R^{\epsilon}$, by setting $C(\epsilon,K)\lesssim C(\epsilon/2,K,D)$ the induction on algebraic tangential term closes.

\subsection{Algebraic transverse case}
In this subsection, we estimate $$\sum_{j}\|e^{it\Delta}f_{j,{\rm trans}}\|_{BL_{k,A/2}^p L^q (B_j)}^p$$ by induction on the radius $R$, where $p=3.2$.

By induction on the radius $R$, we apply \eqref{c} to bound $\|e^{it\Delta}f_{j,{\rm trans}}\|_{BL_{k,A/2}^p L^q (B_j)}$ by
\[
  C(\epsilon,K) \rho^{\epsilon}\rho^{\delta(\log \bar A -\log \frac A 2)}\|f_{j,{\rm trans }}\|_2\,.
\]
Let $\alpha=\rho^{-\frac 12 +\delta_2} $. Note that if $(\theta,\nu)\in \bT_{j,{\rm trans}}$, then $CT_{\theta,\nu}\cap Z_{>\alpha}\cap 2B_j \neq \emptyset$. By Lemma \ref{transverse} (taking radius $ r=R^{\frac 12 +\delta}=\rho^{\frac 12 +\delta_2}$, , so $r\alpha^{-1}=\rho$), we have
\[
   \#\{j\,:\,(\theta,\nu) \in \bT_{j,{\rm trans}}\}\lesssim D^{3} ,\,\,\text{for any}\, (\theta,\nu)\in\bT.
\]
Hence $\sum_j \|f_{j,{\rm trans }}\|_2^2 \lesssim_D \|f\|_2^2$\,, and
\begin{align}
  &\sum_j \|e^{it\Delta}f_{j,{\rm trans}}\|_{BL_{k,A/2}^p L^q (B_j)}^p \notag\\
  \lesssim_D \,& R^{O(\delta)}\left(\frac R\rho\right)^{-\epsilon p} \left[ C(\epsilon,K) R^{\epsilon}R^{\delta(\log \bar A -\log A)}\|f\|_2 \right]^p \,.\notag
\end{align}
Since $\frac R \rho=R^{O(\delta_2)}$, by choosing $\delta \ll \epsilon \delta_2$ the induction closes for the algebraic transverse term. And this completes the proof of part (c) of Proposition \ref{goal}.

\begin{acknowledgement}
The authors wish to express their indebtedness to Larry Guth, who invited them to MIT, 
talked about mathematics with them, and sent his preprint to them.
\end{acknowledgement}

\end{document}